\documentclass[11pt]{article}
\usepackage[left=1in, right=1in, top=1in, bottom=1in]{geometry}

\usepackage{mathtools}
\usepackage{amsmath}
\usepackage{amsthm}
\usepackage{newtxtext}
\usepackage[amsthm]{newtxmath}  % replaces amsfonts, amssymb, txfonts, bbm
\usepackage{dsfont}              % \mathds{1} for indicator functions

\usepackage{xcolor}
\usepackage{geometry}
\usepackage{ulem}
\usepackage{booktabs}
\usepackage{tabularx}
\usepackage{wrapfig}
\usepackage{multirow}
\usepackage{comment}
\usepackage{authblk}
\usepackage{url}
\usepackage{hyperref}

\usepackage{algpseudocode}
\algrenewcommand\algorithmicrequire{\textbf{Input:}}
\algrenewcommand\algorithmicensure{\textbf{Output:}}

\usepackage{tikz}
\usetikzlibrary{matrix,decorations.pathreplacing,shapes,arrows.meta}

\newtheorem{lemma}{Lemma}
\counterwithin{lemma}{section}
\newtheorem{theorem}{Theorem}
\counterwithin{theorem}{section}

\counterwithin{claim}{section}
\theoremstyle{definition}

\counterwithin{definition}{section}

\counterwithin{assumption}{section}

\renewcommand{\tilde}{\widetilde}

\newcommand{\R}{\mathbb{R}}

\newcommand{\st}{\;\;\text{s.t.}\;}

\newcommand{\sfm}{\textrm{{\sc Sfm}}}
\newcommand{\eo}{\textrm{{\sc EO}}}

\providecommand{\keywords}[1]{%
  \small\textbf{\textit{Keywords}} #1%
}
\providecommand{\keywords}[1]
{
  \small	
  \textbf{\textit{Keywords}} submodular optimization, parametric, discrete Newton's method, cutting planes
}

\title{Faster Parametric Submodular Function Minimization by Exploiting Duality}

\author[1]{Swati Gupta}
\author[2]{Alec Zhu\thanks{A part of this work was done when Alec Zhu was an undergraduate student at Massachusetts Institute of Technology.}}
\affil[1]{Massachusetts Institute of Technology, \texttt{swatig@mit.edu}}
\affil[2]{Princeton University, \texttt{az4754@princeton.edu}}
\begin{document}
\maketitle
\vspace{-1cm}

\begin{abstract}
Let $f:2^{E} \rightarrow \mathbb{Z}_+$ be a submodular function on a ground set $E = [n]$, and let $P(f)$ denote its extended polymatroid. Given a direction $d \in \mathbb{Z}^n$ with at least one positive entry, the line search problem is to find the largest scalar $\lambda$ such that $\lambda d \in P(f)$. The best known strongly polynomial-time algorithm for this problem is based on the discrete Newton's method and requires $\widetilde{O}(n^2 \log n)\cdot \sfm$ time, where \sfm\ is the time for exact submodular function minimization under the value oracle model.

In this work, we study the first weakly polynomial time algorithms for this problem. We reduce the number of calls to the exact submodular minimization oracle by exploiting a dual formulation of the parametric line search problem and recent advances in cutting plane methods. We obtain a running time of
  \[
    O\bigl(n^2 \log(nM\|d\|_1)\cdot \eo + n^3 \log(nM\|d\|_1)\bigr) + O(1)\cdot \sfm,
  \]
  where $M = \|f\|_\infty$ and \eo\ is the cost of evaluating $f$ at a set. Note that when $\log \|d\|_1 = O(\log (nM))$, this simplifies to
  \[
    O\bigl(n^2 \log(nM)\cdot \eo + n^3 \log^{O(1)}(nM)\bigr) + O(1)\cdot \sfm,
  \]
  which matches the current best weakly polynomial running time for submodular function minimization~\cite{lsw}, and therefore, one cannot hope to improve this running time. Our approach proceeds by deriving a dual formulation that minimizes the Lovász extension $F$ over a hyperplane intersected with the unit hypercube, and then solving this dual problem approximately via cutting plane methods, after which we round to the exact intersection using the integrality of $f$ and $d$.

\end{abstract}
\keywords{Submodular optimization; Discrete Newton's Method; Cutting Plane Algorithms}

\section{Introduction}
Let $f:2^{E} \to \mathbb{Z}$ be submodular function on a ground set $E = [n]$. For a vector $x\in\mathbb{R}^n$ and $S \subseteq E$, write $x(S) \coloneqq \sum_{i\in S} x_i$. The extended polymatroid and base polytope associated with $f$ are
\begin{align*}
    P(f) &\coloneqq \bigl\{ x\in \mathbb{R}^n : x(S)\leq f(S) \; \forall S\subseteq E \bigr\},\\
    B(f) &\coloneqq \bigl\{ x\in P(f) : x(E)= f(E)\bigr\}.
\end{align*}

Given a point $x_0 \in P(f)$ and an integral direction $d\in \mathbb{Z}^n$ with at least one positive entry (otherwise the movement along $d$ is unbounded), our goal is to find the maximum step length $\lambda^\star$ for which moving from $x_0$ in direction $d$ remains feasible:
\begin{equation}\label{eq:linesearch-general}
  \lambda^\star \;=\;
  \max\bigl\{\lambda \in \mathbb{R}_+ : x_0 + \lambda d \in P(f)\bigr\}.
\end{equation}
Expanding the definition of $P(f)$ and replacing $f^\prime = f-x_0$, this is equivalent to the parametric submodular minimization problem
\begin{equation}\label{eq:parametric}
  \lambda^\star \;=\;
  \max\bigl\{\lambda \in \mathbb{R}_+ : 
  \min_{S \subseteq E}\bigl (f^\prime(S) - \lambda d(S)\bigr) \geq 0 \bigr\},
\end{equation}
where $d(S) \coloneqq \sum_{i\in S} d_i$, and $f^\prime \geq 0$ since $x_0 \in P(f)$. Henceforth, we will assume that $x_0 = 0$ and the submodular function $f$ is non-negative. 

Such parametric problems arise in several algorithmic applications, including line search variants of the Frank-Wolfe method {(e.g., \cite{pedregosa2020linearly, LacosteJulienJaggi2015})}, algorithmic versions of Carath\'eodory's theorem \cite{GrotschelLovaszSchrijver1981}, and {size-constrained densest subgraph problems  \cite{nagano2011size}. In the paper, we will solve the line search problem of the form \eqref{eq:parametric}, and assume the submodular function is integral, with magnitude $M = \|f^\prime\|_{\infty}$. We will further assume that the direction $d \in \mathbb{Z}^n$.} 

\subsection{Related Work} For non-negative directions {$d \geq 0$}, one can solve line search using at most $O(n)$ submodular function minimizations. This is equivalent to a single parametric submodular function minimization with an additional property of inducing ``strong maps"\footnote{A submodular function $\hat{g}$ is said to be a strong quotient of $g$ if $Z \supseteq Y \subseteq E$ implies $g(Z) - g(Y) \geq \hat{g}(Z)- \hat{g}(Y)$. This relation is denoted as $g \rightarrow \hat{g}$, and is referred to as a strong map.}, which allows this problem to be solved in the time of a {\it single} submodular function minimization using combinatorial algorithms {\cite{fleischer2003push,Hochbaum1998_PseudoflowParametric}}. These algorithms exploit the superset structure of minimizers of $f(S) - \lambda d(S)$ as $\lambda$ increases. This structure (or the property of strong maps) breaks down for general directions.

For general directions {$d \in \mathbb{R}^n$}, the first strongly polynomial time algorithm for submodular polytope line search was by Nagano \cite{NAGANO2007349}, where they used Megiddo's parametric search framework; but this algorithm has the prohibitively expensive runtime of $\tilde{O}(n^8)$ \sfm\  calls. The current best strongly polynomial time algorithm analyzes the iterations of the discrete Newton algorithm over the concave envelope $g(\lambda) = \min_{S \subseteq E} f(S) - x_0(S) - \lambda d(S)$. Note that $g(\lambda)$ is piecewise linear, but it is not clear if this envelop has polynomial number of breakpoints. Nevertheless, Goemans et al. \cite{discretenewton} showed that the discrete {Newton's method} will converge in $O(n^2 \log n \cdot \sfm)$ time, by critically analyzing the growth of submodular functions over certain sequences of subsets. These \sfm\ calls need to compute the maximal minimizers for specific breakpoints of $\lambda$. It is unknown if this running time can be improved, {and unclear if more aggressive versions of this method (that use more than 1 \sfm\  in each iteration) will converge faster\footnote{We can construct examples where the submodular grows by a factor of $8^n$ over a sequence of sets $\{S_i\}$ where aggressively changing Newton iterations to $\delta_i^\prime = f(S_i)/(2^n g(S_i))$ instead of $\delta_i = f(S_i)/g(S_i)$ results in $\Omega(n^2)$ iterations over the ${S_i}$ (but we do not have example over the entire ground set $E$).}}. %In fact, we show in Section \ref{} that there exist submodular functions. 

Given upper and lower bounds on $\lambda$, for e.g., $\ell = 0$, and $u = \min_{e| d(e)>0}f(e)/d(e)$, one can simply perform a binary search on $\lambda \in [\ell, u]$ and check for feasibility in each iteration. Though this is folklore, the method was first formalized by Bai et. al \cite{bai2016algorithms}, who showed that using $\log(u/\epsilon)$ calls to \sfm\ (over the function $f - \ell d$) one can find a $(1+\epsilon)$-approximation for maximizing $\lambda$. {Note that this requires an \sfm\ call in each iteration, for this approximate method to get close to the intersection. Our work significantly improves over this baseline, finding the exact intersection in weakly polynomial time.} 

\subsection{Contribution} In this work, we reduce the weakly polynomial running time of line search to $O(n^2 \log (nM\|d\|_1) \cdot \eo + n^3 \log (nM\|d\|_1) + {O(1) \cdot \sfm}$, where \eo\ is the running time for the function evaluation oracle. When $\log(\|d\|_1) \in O(\log(nM))$, then the running time reduces to $O(n^{2} \log(nM) \cdot \mathrm{EO} + n^{3} \log^{O(1)}(nM) + \sfm)$, which is the current best weakly polynomial running time for \sfm\  \cite{lsw}. We first lift the problem to a higher-dimensional base polytope, and then exploit duality to show an equivalent formulation that minimizes the Lov\'{a}sz extension of the submodular function, over a hyperplane intersected with the $[0,1]$ hypercube. This can be approximately solved using the cutting plane machinery, which can further be rounded to the exact intersection, using the integrality of the submodular function and the direction $d$. Instead of using $\tilde{O}(n^2)$ \sfm \ calls in the strongly polynomial method of {\cite{discretenewton}}, we show that one can obtain the exact intersection using a single call to \sfm, and the currently best weakly polynomial time for \sfm\ \cite{lsw}. We summarize the related work and our key result in Table \ref{tab:line-search}. 

\begin{table}[!t]
\centering
\footnotesize
\caption{Algorithms for parametric submodular line search for non-negative directions $d \geq 0$ and general directions $d \in \mathbb{Z}^n$. Here, $L=\log(nM\|d\|_1)$, and \sfm* is the running time of any combinatorial submodular function minimization algorithm.}
\label{tab:line-search}
\begin{tabularx}{\textwidth}{|X|X|c|l|}
\hline
\textbf{Reference} & \textbf{Runtime} & \textbf{Running Time} & \textbf{Assumptions} \\
\hline
Fleischer \& Iwata \cite{fleischer2003push}, Hochbaum \cite{Hochbaum1998_PseudoflowParametric} & $O(1) \cdot \sfm*$ or $O(n) \cdot \sfm$ & Strongly Poly & $d \geq 0$ \\
\hline
Nagano \cite{NAGANO2007349} & $\widetilde{O}(n^8)\cdot \sfm$ & Strongly Poly & General $d$ \\
\hline
Goemans, Gupta, Jaillet \cite{discretenewton} & $\widetilde{O}(n^2) \cdot \sfm$ & Strongly Poly & General $d$ \\
\hline
Bai, Iyer, Wei, Bilmes ~\cite{bai2016algorithms} & $O(\log (u/\epsilon)) \cdot \sfm$ & Weakly Poly & $(1+\varepsilon)$-approx \\
\hline
\textbf{This work} & $\widetilde{O}(n^2 L \cdot \eo + n^3 L) +  O(1) \cdot \sfm$ & Weakly Poly & General $d$ \\
\hline
\end{tabularx}
\end{table}

The paper is organized as follows. We first solve the toy problem of line search over $B(F)$ to illustrate duality in Section \ref{sec:toyproblem}. We then solve the line search problem in Section \ref{sec:solvelinesearch}. We show that the solutions to the line search problem lie on a discrete ladder in Section \ref{sec:rounding}, and hence an $\epsilon$-approximate solution may be rounded to an exact one. We apply cutting plane methods in Section \ref{sec:applycutingplane} to obtain our final runtime.

\section{Preliminaries}

We next introduce preliminaries and background relevant to the paper. Define the Lov\'{a}sz extension $F:\R^n\to\R$ by the following piecewise linear function:
$$F(x) \coloneqq \max_{v\in B(f)} v^\top x.$$
The Lov\'{a}sz extension $F$ is convex and agrees with $f$ on the corners of the unit hypercube. It can be evaluated in $O(n\cdot\eo+n \log n)$ time using the Edmond's greedy algorithm, where \eo\ is the cost of the evaluation oracle for $f(\cdot)$ \cite{fujishige2005submodular}. Its subdifferential is given by
$$\partial F(x) \ni \underset{v\in B(f)}{\arg \max} \; v^\top x.$$
Moreover, the Lov\'{a}sz extension can be written in closed form as
\begin{equation}\label{lovaszdef}
    F(x)=\sum_{i\in[n]} x_{\pi_i}(f(S_i)-f(S_{i-1})),
\end{equation}
where $\pi$ is the permutation induced by $x$ with ties broken lexicographically, i.e. $x_{\pi_1} \geq \cdots \geq x_{\pi_n}$, and $S_i\coloneqq\{\pi_1,\cdots,\pi_i\}$, with $S_0 \coloneqq \emptyset$.

Let $g$ be a convex function. It can be shown that minimizing the composite objective (P1)
\begin{equation}\label{primal}
    \tag{P1} \min_{x\in \R^n} g(x)+F(x)
\end{equation}
is dual to the problem (D2)\footnote{Throughout the rest of the paper, $g$ will be chosen to be a convex piecewise affine function. $F$ is also convex piecewise affine. Since the optimization is done over convex sets with affine constraints, strong duality holds.}:
\begin{equation}\label{dual}
    \tag{D2} \max_{y\in B(f)} -g^*(-y),
\end{equation}
where $g^*(y) \coloneqq \max_{x\in \R^n} y^\top x-g(x)$ is the Fenchel conjugate of $g$ \cite{zhou2018limited}. We will have $g$ be piecewise linear and convex in our setting, and therefore, duality holds and we get: 
\begin{equation}\label{primal}
    \min_{x\in \R^n} g(x)+F(x) = \max_{y\in B(f)} -g^*(-y). 
\end{equation}

We note that the base polymatroids are bounded, and this will be useful in deriving the dual of the line search problem.

\section{Dual of the Line search problem over $B(f)$}\label{sec:toyproblem}
First, we will solve the line search problem over the base polytope $B(f)$, which is bounded\footnote{To see this, note that \(x(e)\le f(e)\) for all $e\in E$, and since \(x(E)=f(E)\), we get $x(e) = x(E) - x(E\setminus\{e\}) = f(E) - x(E\setminus\{e\}) \ge f(E) - f(E\setminus\{e\})$. Therefore, for bounded submodular functions $f$, the base polytope $B(f)$ is also bounded.}.  Assuming that the problem is feasible (i.e., $\bigcup_{\lambda \geq 0} \lambda d \cap B(f) \neq \emptyset$): 
\begin{align}
    \max_{\lambda \in \R_+} \lambda \st \lambda d\in B(f). 
\end{align}
Let $g_R(x)=|R(1-d^\top x)|$ for some positive constant $R$ (we will later set $R = M \|d\|_1$). The Fenchel conjugate of $g_R$ is
$$
    g_R^*(y)\coloneqq \max_x y^\top x - g_R(x) = \max_x y^\top x - |R(1-d^Tx)|.
$$
We first claim that if $y = \lambda d$ with $\lambda \in [-R,R]$, then $g_R^*(y)$  evaluates to $\lambda$, otherwise it is unbounded.

\begin{lemma}\label{lem:fenchel}
For every \(y\in\mathbb{R}^n\),
\[
g_R^*(y)=
\begin{cases}
\lambda, & \text{if } y=\lambda d \text{ for some } \lambda\in[-R,R],\\[4pt]
+\infty, & \text{otherwise.}
\end{cases}
\]
\end{lemma}

\begin{proof}
Write \(y\) as \(y = y_\parallel + y_\perp\) where \(y_\parallel\) is the projection of \(y\) onto the span of \(d\) and \(y_\perp\) is orthogonal to \(d\). Concretely,
\[
y_\parallel = \frac{y^\top d}{\|d\|^2} d,\qquad y_\perp = y - y_\parallel,\qquad y_\perp^\top d = 0.
\]

If \(y_\perp\neq 0\) then the supremum defining \(g_R^*(y)\) is \(+\infty\). Indeed, take any vector \(z\) with \(z^\top d=0\) and \(z^\top y_\perp>0\) (for example \(z=y_\perp\)). For \(t\in\mathbb{R}\) consider \(x_t = t z\). Then \(d^\top x_t = 0\) for all \(t\), hence
\[
y^\top x_t - g_R(x_t) = t\,y^\top z - R|1-0| = t\,(y^\top z) - R.
\]
As \(t\to +\infty\) this tends to \(+\infty\) because \(y^\top z = y_\perp^\top z \neq 0\). Thus \(g_R^*(y)=+\infty\) whenever \(y\) has a nonzero component orthogonal to \(d\). In particular, a necessary condition for finiteness is \(y_\perp=0\), i.e. \(y=\lambda d\) for some scalar \(\lambda\).

{Now assume \(y=\lambda d\). Let \(\delta := d^\top x\) and let \(u := d - d^T x\), i.e., the part of \(x\) in the orthogonal complement of \(d\). Then \(y^\top x = \lambda \delta\) and \(g_R(x)=R|1-\delta|\) depends only on \(\delta\). For any fixed \(\delta\) we may vary \(u\) arbitrarily but \(u\) does not affect the objective, so the supremum reduces to maximizing over the scalar \(\delta\):}
\begin{align}
g^*(\lambda d) \;=\; \sup_{\delta\in\mathbb{R}} \{\lambda \delta - R|1-\delta|\}.
\end{align}
This is a one-dimensional convex piecewise-linear (in \(\delta\)) objective. Compute its value by examining the two linear pieces:
\[
\lambda \delta - R|1-\delta| =
\begin{cases}
\lambda \delta - R(1-\delta) = (\lambda+R)\delta - R, &\text{if } \delta\le 1,\\[4pt]
\lambda \delta - R(\delta-1) = (\lambda-R)\delta + R, &\text{if } \delta\ge 1.
\end{cases}
\]
If \(\lambda \in [-R,R]\), then the slope in the left piece equals \(\lambda+R\ge 0\), and the slope in the right piece equals \(\lambda-R\le 0\). Hence the piecewise-linear function attains its maximum at the kink \(\delta=1\); substituting \(\delta=1\) yields the value \(\lambda\cdot 1 - R\cdot 0 = \lambda\). Thus \(g^*(\lambda d)=\lambda\) for \(\lambda\in[-R,R]\).

If \(\lambda>R\) then the slope of the right piece is \(\lambda-R>0\), so sending \(\delta\to +\infty\) makes the objective unbounded above; similarly if \(\lambda<-R\) the left-piece slope \(\lambda+R<0\) and taking \(\delta\to -\infty\) gives \(+\infty\). Therefore for \(|\lambda|>R\) we have \(g^*(\lambda d)=+\infty\). 

Combining the two cases gives the lemma. 
\end{proof}

Plugging this into the dual (\ref{dual}), we recover the dual problem
\begin{align*}
    \max_{y \in B(f)} -g^*(-y) = %\min_{y\in B(f)} h(y); \; \; h(y)\coloneqq 
    \begin{cases}
       -\lambda \text{ if } y = \lambda d \text{ for } \lambda\in[-R,R] \\
        +\infty \text{ otherwise},
    \end{cases}
\end{align*}

\color{black}
which solves the line search problem as long as the line search intersection $\lambda^\star  \leq R$, {and this can be achieved by setting $R \geq M / \min d_i$}. Thus, for sufficiently large $R$, {the primal problem} $\min_{x\in \R^n} g_R(x) + F(x)$ is equivalent to solving the line search problem. Furthermore, note that $F$ has a bounded subdifferential, and for sufficiently large $R$, the subdifferential of $g_R$ dominates the subdifferential of $F$, so that $g_R(x) + F(x)$ cannot be minimized at points where $g_R(x)$ does not attain its minimum value. Hence, we can solve the line search problem over $B(f)$ by minimizing {the following}: 
\begin{equation}
    \min_{{x \in \mathbb{R}^n}} F(x) \st d^ \top x=1.
\end{equation}

For completeness, we state this as a lemma next: 

\begin{lemma}\label{lem:penalty-active}
{Let \(f\) be nonnegative integer-valued with \(M:=\|f\|_\infty<\infty\), let \(F\) be its Lovász extension, and let \(d\in\mathbb Z^n\) be a nonzero direction. Define the penalized objective
\[
\Phi(x)\;=\;F(x) + g_R(x),\qquad g_R(x):=R\,|1-d^\top x|.
\]
Furthermore, if
\[
R>M,
\]
then every minimizer \(x^\star\) of \(\Phi\) satisfies \(d^\top x^\star = 1\). In particular, under this choice of \(R\) the constrained minimization \(\min\{F(x):d^\top x=1\}\) is equivalent to the unconstrained penalized problem \(\min_x \Phi(x)\).} 
\end{lemma}

\begin{proof}
Note that
$$\partial g_R(x) = \begin{cases}
    \{Rd \cdot \text{sgn}(d^\top x - 1)\} & \text{if}\; 1 - d^\top x\neq  0\\
    \{\nu d : \nu \in [-1,1]\} & \text{if}\;  1 - d^\top x =  0
\end{cases}$$
Any potential minimizer $x^\star $ must satisfy the first order condition $\partial g_R(x^\star ) +\partial F_R(x^\star ) \ni 0$. For any vector $v_F \in \partial F(x^\star )$, $\| v_F \|_\infty \leq M$. If $d^\top x^\star  \neq 1$, then the only subgradient of $g_R$ at $x^\star $ is $Rd\cdot \text{sgn}(d^\top x- 1)$, which has $\infty$-norm $\geq R$ since $d$ is integral. Thus, any $x^\star $ with $d^\top x^\star \neq 1$ could not possibly minimize the penalized objective.
\end{proof}

As a simple example, let $\{1,2\}$ be the ground set, and $f(\emptyset)=0$, $f(\{1\})=f(\{2\})=2$, and $f(\{1,2\})=3$. The vertices of $B(F)$ are therefore $(1,2)$ and $(2,1)$. We consider the line search problem with $d = (3, 4)$. As established above, we minimize the Lov\'{a}sz extension with the constraint that $d^\top x=1$. The Lov\'{a}sz extension is
$$F(x)=\max\{2x_1+x_2, x_1+2x_2\}$$
and we wish to solve $$\min F(x) \st 3x_1+4x_2=1$$
The minimum is achieved at $x^\star =(1/7,1/7)$, and an objective value of $F(x^\star )=3/7$ is achieved, which is the same as the line search solution $\lambda^\star d = (9/7,12/7)\in B(f)$.

\section{Line search over $P(F)$}\label{sec:solvelinesearch}
Now we solve the line search problem over the extended polymatroid $P(f)$: 
\begin{equation}
    \max_{\lambda \in \R_+} \lambda \st \lambda d\in P(f),
\end{equation}
{by mapping the problem to an equivalent line search problem over a base polytope in a higher dimension.} 
%The idea is to lift $P(f)$ by one dimension so that it becomes the base polytope of a higher-dimensional polytope. 
This is not directly possible, since $P(f)$ is unbounded in the negative direction, but the base polytope is always bounded. However, the intersection point of {the line search $\lambda^\star  d$ must be bounded, since we assumed that} $d$ contains one positive coordinate. {We will instead} consider line search over a truncated $P(f)$ that is known to contain the intersection point, so that the truncated $P(f)$ corresponds to some base polytope. 

Let $C \gg M$ be a large positive constant. Now we define a submodular function over the ``lifted" ground set $\hat{E}= {E \cup \{n+1\}}$ parameterized by $C$:
\begin{equation} \label{eq:lifted}
\hat{f}:2^{\hat{E}} \times \R\to \R_+, \quad \hat{f}(S;C)=\begin{cases}
    f(S) &\text{ if } \{n+1\} \not\subset S, \\
    f(S \setminus \{n+1\}) +C &\text{ if } \{n+1\} \subseteq S\subsetneq \hat{E}, \\
    f(E) &\text{ if } S = \hat{E}.
\end{cases}
\end{equation}

We first note that $\hat{f}$ is submodular. First, for any subsets $S, T \subset \hat{E}$ such that either (i) $S \cup T \not\supseteq \{n+1\}$, or (ii) $S\cap T \supseteq \{n+1\}$ and $S \cap T \neq \hat E$, $\hat{f}(S\cap T) + \hat{f}(S \cup T) \leq \hat{f}(S) + \hat{f}(T)$, due to submodularity of $f$. Otherwise, if $S\cap T \supseteq \{n+1\}$ and $S \cup T = \hat E$, then $\hat{f}(S\cap T) + \hat{f}(S \cup T) = f(S \cap T) + C + f(E) \leq f(S) + f(T) + C \leq \hat{f}(S) + \hat{f}(T).$ The remaining case is if $S \setminus T \supseteq \{n+1\}$, then $\hat{f}(S) + \hat{f}(T) = f(S) + C + f(T) \geq f(S \cap T) + f(S \cup T) + C = \hat{f}(S \cap T) + \hat{f}(S \cup T)$.

Note that as $C \rightarrow \infty$, $B(\hat{f}(\cdot;C))$ becomes arbitrarily large. Let $x_{[n]}$ denote the prefix of $x$ of length $n$. The lifting induces a canonical correspondence between points $x\in B(\hat{f})$ and points in $P(f)$; namely, the map $\psi:(x_1,\cdots,x_n,x_{n+1}) \mapsto (x_1,\cdots,x_n)$.
Let $\psi(B(\hat{f})) \subset P(f)$ be the image of this map. If we set $C$ sufficiently large, then the optimum for the line search problem will lie in $\psi(B(\hat{f}))$, which means that it suffices to solve a ``lifted" line search problem over $B(\hat f)$. We state this precisely in the following lemma.

\begin{lemma} \label{lem:lifting-equivalence}
    Let $f:2^E\rightarrow \mathbb{Z}$ be an integer valued submodular function with $M \coloneqq \|f\|_\infty <\infty$, and $\hat f(\cdot; C)$ be the lifted submodular function over $\hat E = E \cup \{n+1\}$ be defined {above (see \ref{eq:lifted})}. Fix a non-zero integer direction $d\in \mathbb{Z}^n$, and let $\hat d = d \oplus 0 \in \mathbb{Z}^{n+1}$. For $C > M\|d\|_1$,
\begin{align}\max_{
    \substack{\lambda_1, \lambda_2 \in \R\\ \lambda_1 \hat{d} + \lambda_2 e_{n+1} \in B(\hat{f}(\cdot;C))}
} \lambda_1 
= \max_{\substack{\lambda d\in P(f) }} \lambda. \label{eq:obj}
\end{align}
\end{lemma}

\begin{proof}
    Let $z^\star = \lambda^\star d \in \R^n$ be the optimal solution to the original line search problem $\max_{\lambda d \in P(f)}\lambda$ attaining objective $\lambda^\star$. Let $x^\star =\lambda_1^\star \hat d + \lambda_2^\star e_{n+1} \in \mathbb{R}^{n+1}$ be the optimal solution to 
    \begin{align}
    \max_{\lambda_1, \lambda_2 \in \R, \lambda_1 \hat{d} + \lambda_2 e_{n+1} \in B(\hat{f}(\cdot;C))} \lambda_1, \label{eq:l1l2}
    \end{align}
attaining objective $\lambda_1^\star$. $x^\star \in B(\hat f(\cdot; C))$ corresponds to $x^\star_{[n]} \in P(f)$ which obtains the same objective value, which shows that $\lambda_1^\star \leq \lambda^\star$. Now to show that $\lambda^\star_1 \geq \lambda^\star$, $z^\star \in P(f)$ corresponds to $\hat z =(z_1^\star,\cdots, z_n^\star,f(E) - \sum_i z^\star_i) \in B(\hat f(\cdot; C))$. Note that $\hat z$ is feasible for $B(\hat f(\cdot; C))$ if $\hat z (S) \leq \hat f(S;C)$ for all $S \subseteq \hat E$, which holds true if $C > M\|d\|_1 \geq \|z^\star\|_1$. This shows the equivalence of the two problems, as $\lambda^\star = \lambda_1^\star$. \qedhere

    %Then, $z^\star \in P(f)$ corresponds to $x = $(z_1^\star,\cdots, z_n^\star,f(E) - \sum_i z^\star_i) \in B(\hat f(\cdot; C))$$; therefore, $\lambda_1^\star \geq \lambda^\star$.} 
    
    %{Further, since $x^\star$ is feasible for $B(\hat f(\cdot; C))$, $x^\star (S) \leq \hat f(S;C)$ for all $S \subseteq \hat E$, which holds true if $C > M\|d\|_1 \geq \|z^\star\|_1$. Further, note that $(x_1^\star, \hdots, x_n^\star) \in P(f)$ since $f(S) = \hat{f}(S)$ if $n+1 \notin S$. This shows the equivalence of the two problems, as $\lambda^\star \geq \lambda_1^\star$.} 
\end{proof}

Now we will derive the dual of the lifted line search problem.

\begin{lemma} \label{lem:lifted-conjugate}
    Define the penalty $h_R(x) = R|1 - d^\top x_{[n]}| + R|x_{n+1}|$, where $x_{[n]}$ denotes the $n$-prefix of $x$. Let $\hat F(\cdot ;C)$ be the Lov\'{a}sz extension of $\hat f(\cdot ; C)$. Then
    $$\max_{
    \substack{\lambda_1, \lambda_2 \in \R\\ \lambda_1 \hat{d} + \lambda_2 e_{n+1} \in B(\hat{f}(\cdot;C))}
} \lambda_1 = \min_{x \in \R^{n+1}} \hat F(x;C) + h_R(x).
    $$
\end{lemma} 

%{\color{red} this is a cleaner lemma, shorter. just deriving the dual. I've moved the rest to Lemma 4.3. }

\begin{proof}

{The Fenchel conjugate of the function }
$h_R(x)\coloneqq R|1-d^\top x_{[n]}| + R|x_{n+1}|$ {is given by:}
\begin{align*}
    h_R^*(y)&=\max_x \left\{  y_{[n]}^\top x_{[n]} + y_{n+1}x_{n+1} - R|1-d^\top x_{[n]}| - R|x_{n+1}| \right\} \\
    &= \max_{x_{[n]}} \left\{  y_{[n]}^\top x_{[n]} - R|1-d^\top x_{[n]}|  \right\} + \max_{x_{n+1}}\{ y_{n+1}x_{n+1} - R|x_{n+1}| \}.
\end{align*}
The first maximization term is exactly the same as that described in Lemma \ref{lem:fenchel}, and the second one is bounded for $y_{n+1} \in [-R,R]$ and takes on a value of $0$ at $x_{n+1}=0$. Hence,
\begin{equation}
    h_R^*(y)=\begin{cases}
        \lambda \text{ if } y_{[n]} = \lambda d \text{ for } \lambda \in [-R,R] \text { and } y_{n+1}\in[-R,R], \\
        +\infty \text{ otherwise}.
    \end{cases}
\end{equation}

{This implies that $\max_{B(\hat{f})} -h_R^*(-y)$ solves the \eqref{eq:l1l2} line search problem over $B(\hat{f})$ for $R$ sufficiently large, e.g., $R \geq M \|d\|_1$.}
\end{proof}

%{\color{red} second part of the lemma:}

\begin{lemma}\label{lem:3} Let $f:2^E\rightarrow \mathbb{Z_+}$ be an integer valued submodular function with $M \coloneqq \|f\|_\infty <\infty$, and $\hat f(\cdot; C)$ be the lifted submodular function over $\hat E = E \cup \{n+1\}$ be defined in \eqref{eq:lifted}. Let the penalized objective $\Phi(x) = \hat F(x;C) + h_R(x)$ be defined for $C > M \|d\|_1$. If $R > C + M$, then every minimizer $x^\star$ of $\Phi$ satisfies
    $$d^\top x^\star_{[n]} = 1,\qquad x^\star_{n+1}=0.$$

The penalty $h_R$ vanishes along these constraints, so we have for sufficiently large $R$
\begin{equation}\label{eq:thmduality}
 \min_{x \in \R^{n+1}} \hat F(x;C) + h_R(x) = \min_{\substack{x \in \R^{n+1},\\ d^\top x_{[n]}=1,\\ x_{n+1}=0}} \hat F(x;C). 
\end{equation}

\end{lemma}

\begin{proof}
This follows via similar arguments of subgradients being dominated, as in Lemma \ref{lem:penalty-active}. Namely, for any $v_F\in\partial \hat F(x;c)$, $\|v_F\|_\infty \leq \|\hat f (\cdot ,C)\|_\infty \leq M+C$, but for any x with $d^\top x_{[n]} \neq 1$ or $x_{n+1} \neq 0$, $\|v_{h_R}\|_\infty \geq R$ for any $v_{h_R} \in \partial h_R(x)$, and so the first order optimality condition could not be satisfied.
\end{proof}

Next, we prove the main result of this section by showing that the optimization of $\hat F(\cdot; C)$ may be restricted to the first orthant, thus deriving the final form of the dual.

\begin{theorem}[Line search duality]\label{thm:linesearchdual} {For any submodular function $f: 2^E \rightarrow \mathbb{R}_+$, with the Lov\'{a}sz extension $F(\cdot)$, and a non-zero direction $d \in \mathbb{R}^n$, the line search problem is dual to minimization of the Lov\'{a}sz extension intersected with the hyperplane $d^Tx = 1$ in the positive orthant, i.e.,} 
\begin{equation}\label{eq:thmduality}
\max_{\substack{\\ \lambda d \in P(f)}} \lambda
=\min_{\substack{x\in\R^n_+, \\ d^\top x= 1.}} F(x). 
\end{equation}
\end{theorem}
\begin{proof}
So far, using Lemmas \ref{lem:lifting-equivalence}, \ref{lem:lifted-conjugate} and \ref{lem:3} with $C \geq M\|d\|_1$, $R \geq  C + M$, we have shown: 
\begin{equation}\label{eq:thmduality}
\max_{\substack{\\ \lambda d \in P(f)}} \lambda
\overset{(\ref{lem:lifting-equivalence})}{=} \max_{
    \substack{\lambda_1, \lambda_2 \in \R,\\ \lambda_1 \hat{d} + \lambda_2 e_{n+1} \in B(\hat{f}(\cdot;C)).}
} \lambda_1 \overset{(\ref{lem:lifted-conjugate})}{=} \min_{x \in \R^{n+1}} \hat F(x;C) + h_R(x) \overset{(\ref{lem:3})}{=} \min_{\substack{x \in \R^{n+1},\\ d^\top x_{[n]}=1,\\ x_{n+1}=0}} \hat F(x;C). 
\end{equation}

We need to show for $C$ sufficiently large,
\begin{align}
\min_{\substack{x \in \R^{n+1},\\ d^\top x_{[n]}=1,\\ x_{n+1}=0}} \hat F(x;C) = \min_{\substack{x\in\R^n_+, \\ d^\top x= 1.}} F(x). \label{eq:49}
\end{align}
We will now set $C = M^2 \|d\|_1$ and show \eqref{eq:49} by proving that any $x \notin \R_+^{n+1}$ cannot be optimal, by analyzing the subdifferential $\partial \hat F(x;C)$. First, due to our choice of $C$, unless $n+1\notin S$ or $S = \hat E$, $\hat f(S;C) = C + O(M)  \gg M$. If $n+1\in S$, $\hat f(S;C) \leq M$. The optimum $x^\star$ to (\eqref{eq:49}, left hand side) induces a permutation $\pi$ and sets $S_i$ as defined in equation (\ref{lovaszdef}). Now suppose that $x^\star$ had only $k<n+1$ non-negative coordinates, i.e. $$x^\star _{\pi_1} \geq \cdots \geq x^\star _{\pi_{k}} = 
x^\star _{n+1}=0>x^\star _{\pi_{k+1}}\geq \cdots \geq x^\star _{\pi_{n+1}}.$$
Then we have
\begin{align} \label{eq:subgradient-norm}
\hat{F}(x^\star ;C) &=\sum_{i=1}^{k-1}x^\star_{\pi_i}
\underbrace{(\hat{f}(S_i) - \hat{f}(S_{i-1}))}_{= O(M)} + \underbrace{x^\star_{\pi_{k}}(\cdots)}_{=0} + 
\sum_{i=k+1}^{n}x^\star_{\pi_i}
\underbrace{(\hat{f}(S_i) - \hat{f}(S_{i-1}))}_{= O(M)} \\
&+x^\star_{\pi_{n+1}}
\underbrace{(\hat{f}(S_{n+1})- \hat{f}(S_n))}_{=\hat{f}(\hat{E})-\hat{f}(S_n) \leq -C + M}. \nonumber
\end{align}
We see that $|\partial \hat F(x^\star ;C)_{\pi_i}| \leq M$ for all coordinates $\pi_i$ except for the last coordinate in the permutation $\pi_{n+1}$. Thus, for any given $d$, we can choose $C$ sufficiently large so that the subgradient on the last coordinate is sufficiently negative and always dominates the other coordinates, so that we can improve the objective by raising $x_{\pi_{n+1}}^\star$, contradicting its optimality. Therefore, $x^\star$ cannot have a negative coordinate. Note that $\hat F(x;C)=F(x_{[n]})$ if $x \geq 0$ and $x_{n+1} = 0$, so the problem $\min_{x:x_{n+1} = 0} \hat{F} (x;C)$ reduces to minimizing the original Lov\'{a}sz extension $\min_{x:d^\top x = 1, x \in \R_+^n} F(x)$, which concludes the proof.
\qedhere
\end{proof}

For a simple example, reconsider the example from section \ref{sec:toyproblem}; i.e. let $\{1,2\}$ be the ground set, and $f(\emptyset)=0$, $f(\{1\})=f(\{2\})=2$, and $f(\{1,2\})=3$. Consider the line search direction $d = [1,-1]^\top$. As before, the Lov\'{a}sz extension is
$F(x)=\max\{2x_1+x_2, x_1+2x_2\}$,
and this time, we want to minimize $F$ over the set $\{x_1, x_2\in \R_+:x_1-x_2=1\}$.
The dual minimum is achieved at $x^\star =(1,0)$, and an objective value of $F(x^\star )=2$ is achieved, which is the same as the line search solution $\lambda^\star = 2$ with $\lambda^\star d = (2,-2)\in P(f)$.

\section{Computing the exact intersection via rounding}\label{sec:rounding}

The line search intersection occurs at
$$\lambda^\star  = \max \left\{ \lambda : f(S) - \lambda d(S) \geq 0 \; \forall S\subseteq E \right\}.$$
The exact intersection can be computed using discrete Newton's method \cite{discretenewton}, which finds a zero of the piecewise affine concave envelope
\begin{equation} \label{envelope}
    g(\lambda) \coloneqq \min_{S \subseteq E} f(S)-\lambda d(S).
\end{equation}
The discrete Newton's method, given an initial starting point $\lambda_0\geq \lambda^\star $, repeatedly iterates $S_i = \arg \min_{S \subseteq E} f(S) -\lambda_id(S)$, and $\lambda_{i+1} = f(S_i)/d(S_i)$. The algorithm terminates when $g(\lambda_i)=0$. Note that the iterations are strictly decreasing until termination.

The discrete Newton's method takes $\tilde O (n^2)$ iterations, where each iteration requires one $\sfm$ call. However, a natural question is if one can do better with side information about the range of $\lambda^\star $. We show that the iterates of the discrete Newton's method lie on a discrete ladder, with spacing of the ladder say $\epsilon$. This implies that given a $\lambda_0$ such that $\lambda_0-\lambda^\star  \leq O(\epsilon)$, the discrete Newton's algorithm will converge in $O(1)$ iterations.

Consider the set $\Lambda=\{\lambda_S : S \subseteq E\}$, where $\lambda_S=f(S)/d(S)$ if $d(S) > 0$, and $+\infty$ otherwise. Then $\lambda^\star =\min_\Lambda \lambda$. Note that this set contains all the possible iterates of the discrete Newton's method, as $\lambda_{i+1} = f(S_i)/d(S_i) \in \Lambda$. For any two sets $S_1,S_2$ with $\lambda_{S_1} \neq \lambda_{S_2}$, we must have
$$|\lambda_{S_1} - \lambda_{S_2}|=
\left| \frac{f(S_1)}{d(S_1)} - \frac{f(S_2)}{d(S_2)} \right|
=\left| \frac{f(S_1)d(S_2) - f(S_2)d(S_1)}{d(S_1)d(S_2)}\right|
\geq \frac{1}{\|d\|_1^2} = \epsilon,$$
using the integrality assumption of $f(\cdot)$ and $d$. If we can obtain an $\epsilon$-approximation $\lambda_\epsilon$ to $\lambda^\star $, and initialize the discrete Newton's method with $\lambda_0 = \lambda_\epsilon$, then the first iteration of this method must converge exactly to $\lambda^\star $. Otherwise, there would be an iterate $\lambda_1 \in \Lambda$ with $\lambda^\star  < \lambda_1 < \lambda_0$ with the spacing $\lambda_1-\lambda^\star  < \epsilon$. More generally, if the discrete Newton's method is initialized with $\lambda_0 - \lambda^\star  \leq k\epsilon$, the algorithm must converge within $k$ iterations, i.e., $O(\epsilon)$ gap requires at most $O(1)$ iterations. 

%Suppose we use some approximate algorithm to obtain $\lambda_\epsilon$, a $\epsilon$-approximation to $\lambda^\star $, with $\epsilon = \|d\|_1^{-2}$, and we have $\lambda_\epsilon \geq \lambda^\star $. Then a minimizer for the submodular function $f'(S)=f(S)-\lambda_\epsilon d(S)$ will be a tight set for the line search problem. Suppose $T$ was a minimizer for $f'$, but was not a tight set for the line search problem. Then $\lambda_\epsilon > \lambda_{T} > \lambda^\star $, but then $|\lambda_{T}-\lambda^\star | < \epsilon$, which is a contradiction. Thus, we must have $\lambda_T=\lambda^\star $. Thus, at the cost of one additional SFM, we can recover an exact solution to the line search problem from an $\epsilon$-approximate solution, with $\epsilon=\|d\|_1^{-2}$.

We next show that this analysis is tight. That is, there exist instances where the discrete Newton's method requires initialization within $\lambda_0 - \lambda^\star  = O(\epsilon)$ for convergence within one iteration. %The $1/\|d\|_1^2$ scaling required to solve the line search problem in a single iteration of Newton's algorithm is tight, as shown by the following example. 
Consider intervals $[i,j] \coloneqq \{k \in E:i\leq k \leq j\}$, and the function
$$h:[i,j]\mapsto 4^{j(j-1)/2}4^i.$$
The function $h$ can be extended to a submodular function $f:2^E\rightarrow \mathbb{Z}$ by Lemma 8 of \cite{discretenewton}, whose result we briefly restate. Let $\mathcal{I}(S)$ be the set of maximum intervals contained in $S$; for example, $\mathcal{I}(\{1,2,3,6,9,10\})=\{[1,3],[6,6],[9,10]\}$. Then $f(S)\coloneqq\sum_{I\in \mathcal{I}(S)}h(I)$ is submodular over $E$. Since $f$ exhibits geometric growth, parts of the $\tilde O(n^2)$ runtime analysis for discrete Newton's algorithm are tight \cite{discretenewton}.

Now consider the line search direction $d=[D,3D-1,d_3,\cdots,d_n]^\top$, where $D \gg d_i$. Consider the breakpoints of the concave envelope $g(\lambda)$ as defined in \ref{envelope}. Note that $f$ takes on minimum/second-minimum values $f(\{1\}) = 4$ and $f(\{1,2\})=16$ respectively, and thus the first two segments of this envelope are $4-\lambda D$ and $16-\lambda(4D-1)$. The first breakpoint therefore occurs at $\lambda = 12/(3D-1)$. In order for discrete Newton's method to converge in exactly one iteration, we need to initialize it before the first breakpoint occurs, i.e., choose $\lambda_0 \in [\lambda^\star , 12/(3D-1)]$. Since $\lambda^\star  = 4/D$, and $12/(3D-1) - 4/D\sim 4/(3D^2)$, the $1/\|d\|_1^2$ scaling is required for a non-trivial example.

\section{Applying Cutting Plane Methods}\label{sec:applycutingplane}
Let $B_p(x,r)$ denote the closed $L_p$-ball of radius $r$ centered at $x$. Define the minwidth of a (closed) set $\Omega\subset \R^n$ to be
$$\text{minwidth}(\Omega) \coloneqq \min_{a\in \R^n:\|a\|_2=1} \left( \max_{y\in \Omega} a^\top y - \min_{y \in \Omega} a^\top y\right).$$
With these definitions, we have the following theorem by Jiang et al \cite{jiang2020improvedcuttingplanemethod}.

\begin{theorem}[Cutting Plane Running Time \cite{jiang2020improvedcuttingplanemethod}] \label{runtime}
Let $F$ be a convex function on $\R^n$ with a subgradient oracle $\partial F$ with cost $\mathcal{T}$. Let $\Omega$ be a convex set containing a minimizer ($x^\star $) of $F$, with $\Omega\subset B_\infty(0,\mathcal{R})$. Using $B_\infty(0,\mathcal{R})$ as the initial polytope for the Cutting Plane Method, for any $\alpha\in(0,1)$ we can compute $x\in S$ such that $F(x)-F(x^\star )\leq \alpha (\max_{y \in \Omega} F(y) - F(x^\star ))$ with a running time of
\begin{align}O(\mathcal{T}\cdot n \log \kappa + n^3 \log \kappa), \quad 
\text{ where } \quad 
\kappa= \frac{n\mathcal{R}}{\alpha \cdot \textnormal{minwidth}(\Omega)}.
\end{align}
\end{theorem}

To apply the cutting plane method, note that with the constraint $d^\top x=1$, we have an $n-1$ dimensional problem. Without loss of generality, assume $d_n >0$, and let $\Omega\coloneqq \{z\in \R_+^{n-1} : d^\top _{[n-1]}z \leq 1 \}$. Given $z\in \Omega$, let $\zeta_n(z) = d_n^{-1}(1-d_{[n-1]}^\top z) \geq 0$. Define $\phi(z) \coloneqq F(z\oplus \zeta_n(z))$. Minimizing $\phi$ over $\Omega$ using the cutting plane method solves the line search problem (see \eqref{eq:thmduality}). We now give runtime bounds for applying cutting plane methods to this problem.

\begin{theorem} {Given a submodular polytope defined by an integral submodular function $f \geq 0$}, with $\|f\|_\infty = M$ and Lov\'{a}sz extension $F(\cdot)$, and an integral line search direction {$d \in \mathbb{Z}^n$} with at least one positive component, using cutting plane methods the dual of the line search problem, {$\min_{\{x: d^\top x = 1, x\geq 0\}} F(x)$} can be solved in:
\begin{equation}
    O(\textnormal{EO}\cdot n^2 \log \kappa + n^3 \log \kappa) + O(1)\cdot \sfm
\end{equation}
runtime, with $\log \kappa = \log (nM\cdot \|d\|_1)$, where \textnormal{EO} is the cost of the evaluation oracle for $f$.    
\end{theorem}
\begin{proof}
    We assume without loss of generality that $d_n>0$. First, we bound $\text{minwidth}(\Omega)$, where $\Omega\coloneqq \{z\in \R_+^{n-1} : d^\top _{[n-1]}z \leq 1 \}$. The vectors $\textbf{0}, \textbf{1}^{n-1}/\|d\|_1\in \Omega$, so the box $\{z:z\leq \textbf{1}^{n-1}/\|d\|_1\}\subset \Omega$. Hence, $\text{minwidth}(\Omega) \geq 1/\|d\|_1$. 

    Next, we show that we may restrict our search over $\Omega$ to an $L_\infty$-ball, by showing that an approximate minimizer can be found in a bounded domain. Consider a perturbation $f_{\epsilon} \coloneqq f + \epsilon$, for $\epsilon = 1/\|d\|_1^2$. It is submodular, though not integral any more (but integrality is not required for the cutting plane method). Consider an ordering $\pi$ of subsets $S \subseteq E$ in decreasing order of $\lambda_S = f(S)/d(S)$. The line search problem seeks to find the minimum $\lambda_S$, over subsets with $d(S)>0$. Note that the perturbed function $f_{\epsilon}$ induces the same ordering of subsets. Therefore, the optimal solution $\lambda_{\epsilon}^\star$ over $f_{\epsilon}$ is also $O(\epsilon)$-approximate for $\lambda^\star$, i.e, $\lambda_{\epsilon}^\star \leq \lambda^\star + O(\epsilon)$.   
 
    We next show that the minimizer of Lov\`{a}sz extension $F_\epsilon$ of $f_\epsilon$ over the first orthant must lie within $B_\infty(0, 2M/\epsilon)$, and therefore $B_\infty(0, \mathcal{R})$ may be used as the starting polytope for the cutting plane method, with $\mathcal{R} = 2M/\epsilon$.
    
    %Suppose that we perturb $f$ by some small positive amount $\delta$; setting $\delta = \epsilon/2$ with $\epsilon=\|d\|_1^{-2}$ {\color{red} these orderings are not defined. also changing $f$ be $\delta$ is very tricky, as this might break submodularity.} will not change the tight sets or the orderings of the $\lambda_S$. Thus, an $\epsilon/2$-approximate solution for the perturbed problem is an $\epsilon$-approximate solution for the original problem. We show that an optimizer to the perturbed problem has bounded $L_\infty$-norm $\|x^\star \|_\infty \leq \mathcal{R} \coloneqq M/\delta$.

%{\color{red} why don't I see the choice of $C$ here?} 

Consider the value of $F_\epsilon(x)$ at some $x$ with $\|x\|_\infty > 2M/\epsilon$:
$$F_\epsilon(x)=\|x\|_\infty\cdot F_\epsilon(\underbrace{x/\|x\|_\infty}_{\in [0,1]^n}) \overset{(a)}{>} (2M/\epsilon) \cdot \epsilon = 2M,$$
where (a) follows since $\epsilon$ is the minimum of the perturbed function $f_{\epsilon}$ over the unit hypercube. If such an $x$ were optimal, this would imply $\lambda_{\epsilon}^\star > 2M$. But since $\lambda^\star \leq M$ for the original line search problem, this is impossible (since $M \in \Omega(\epsilon)$).

%Since we must have $\lambda^\star \leq M$, such an $x$ cannot be optimal for the perturbed problem. Thus, we need only optimize over $x$ with bounded $L_\infty$ norm.

We will now consider the cutting plane method over $\Omega' = \Omega \cap B_\infty(0, \mathcal{R})$. Compute
$$\max_{z\in\Omega'} \phi(z) -\phi(x^\star ) \leq \max_{z\in\Omega'} \|z\|_\infty \cdot M = 2M^2/\epsilon,$$
so we can set $\alpha = \epsilon/M^2$ as needed in Theorem \ref{runtime}. Plugging $\mathcal{R}=2M/\epsilon$, $\text{minwidth}(\Omega)\geq \|d\|_1^{-1}$, and $\alpha = \epsilon/(2M^2)$ into \ref{runtime}, we require
$$\kappa \asymp \frac{nM^3\cdot \|d\|_1}{\epsilon ^3},$$
to achieve $O(\epsilon)$ error in $O(\mathcal{T}\cdot n \log \kappa + n^3 \log \kappa)$ runtime. After the approximation using the cutting plane method, one can round the solutions to the exact optimum $\lambda^\star$ only $O(1)$ $\sfm$. Finally, the subgradient of $\phi$ is computed using the Edmond's greedy algorithm for the Lov\'{a}sz extension, which requires $n\cdot \eo + O(n \log n)$ runtime. This completes the proof.
\end{proof}

\section*{Acknowledgments}
The authors would like to thank Jai Moondra for discussions about rounding to the exact intersection.

\bibliographystyle{siamplain}
\bibliography{references}

%%%%%%%%%%%%%%%%%
\end{document}